\documentclass{ieeeaccess}
\usepackage{cite}
\usepackage{amsmath,amssymb,amsfonts}
\usepackage{amsthm} %
\usepackage{algorithmic}
\usepackage{graphicx}
\usepackage{textcomp}

\newtheorem{theorem}{Theorem}[section]
\newtheorem{lemma}{Lemma}[section]

\newtheorem{corollary}{Corollary}

\allowdisplaybreaks

\def\BibTeX{{\rm B\kern-.05em{\sc i\kern-.025em b}\kern-.08em
    T\kern-.1667em\lower.7ex\hbox{E}\kern-.125emX}}
\begin{document}
\history{Date of publication xxxx 00, 0000, date of current version xxxx 00, 0000.}
\doi{10.1109/ACCESS.2023.0322000}

\title{Net-Zero Energy House-oriented Linear Programming for the Sizing Problem of Photovoltaic Panels and Batteries}
\author{\uppercase{A. Daniel Carnerero}\authorrefmark{1}, \uppercase{Taichi Tanaka}\authorrefmark{1}, \uppercase{Mengmou Li}\authorrefmark{2}, \uppercase{Takeshi Hatanaka}\authorrefmark{1}, \uppercase{Yasuaki Wasa}\authorrefmark{3}, \uppercase{Kenji Hirata}\authorrefmark{4},
\uppercase{Yoshiaki Ushifusa}\authorrefmark{5}, and \uppercase{Takanori Ida}\authorrefmark{6}
}

\address[1]{Department of Systems and Control Engineering School of Engineering, Tokyo Institute of Technology, Tokyo, 152-8550, Japan (e-mail: \{carnerero; tanaka\}@hfg.sc.e.titech.ac.jp, hatanaka@sc.e.titech.ac.jp).}
\address[2]{Graduate School of Advanced Science and Engineering, Hiroshima University, Higashi-Hiroshima 739-0046, Japan (mengmou@hiroshima-u.ac.jp).}
\address[3]{Department of Electrical Engineering and Bioscience, School of Advanced Science and Engineering, Faculty of Science and Engineering, Waseda University, Tokyo 169-8555, Japan (e-mail: wasa@waseda.jp)}
\address[4]{Department of Electrical and Electronic Engineering, School of Engineering, University of Toyama, Toyama 930-8555, Japan (e-mail: hirata@eng.u-toyama.ac.jp)}
\address[5]{Faculty of Economics and Business Administration, The University of Kitakyushu, Fukuoka 802-8577, Japan (e-mail: ushifusa@kitakyu-u.ac.jp).}
\address[6]{Graduate School of Economics, Kyoto University, Kyoto 606-8501, Japan (e-mail: ida@econ.kyoto-u.ac.jp).}
\tfootnote{This work is supported by the Ministry of the Environment, Government of Japan, and JSPS KAKENHI under Grant 21J13956.}

\markboth
{A.D. Carnerero \headeretal: Zero Energy House-oriented Linear Programming for the Sizing Problem of Photovoltaic Panels and Batteries}
{A.D. Carnerero \headeretal: Zero Energy House-oriented Linear Programming for the Sizing Problem of Photovoltaic Panels and Batteries}

\corresp{Corresponding author: A. Daniel Carnerero (e-mail: carnerero@hfg.sc.e.titech.ac.jp).}

\begin{abstract}
The global drive towards carbon neutrality has led to a significant increase in the number of power plants based on renewable energy sources (RES). Concurrently, numerous households are adopting RES to generate their own energy, aiming to decrease both electricity costs and carbon footprints.
To support these users, many papers have been devoted to developing optimal investment strategies for residential energy systems. However, there is still a significant gap as these studies often neglect important aspects like carbon neutrality.
For this reason, in this paper, we explore the concept of net-zero energy houses (ZEHs)---houses designed to have an annual net energy consumption around
zero---by presenting a constrained optimization problem to find the optimal number of photovoltaic panels and the optimal size of the battery system for home integration.
Solving this constrained optimization problem is difficult due to its nonconvex constraints.
Nevertheless, by applying a series of transformations,
we reveal that it is possible to find an equivalent linear programming (LP) problem which is computationally tractable. The attainment of ZEH can be tackled by introducing a single constraint in the optimization problem. Additionally, we propose a sharing economy approach to the investment problem, offering a strategy that could potentially reduce investment costs and facilitate the attainment of ZEH more efficiently. Finally, we apply the proposed frameworks to a neighborhood in Japan as a case study, demonstrating the potential for long-term ZEH attainment. The results show that, under the right incentive, users can achieve ZEH, reduce their electricity costs and have a minimal impact on the main grid.
\end{abstract}


\titlepgskip=-21pt

\maketitle

\section{Introduction}
As the global economy recovers from the impacts of COVID-19, energy consumption has resumed its upward trend. 
In 2022, global electricity demand saw an increase of $2.7\%$ compared to the previous year \cite{IEA2023}.
Meanwhile, the recent worldwide energy crisis accelerates a shift from reliance on fossil fuels towards renewable energy sources as well as adopting more efficient energy consumption schedules.
For example, \cite{yan2022low} proposes an optimal carbon-dispatching method, which allows entities to optimize their operation costs and carbon emissions.
In this context, the role of buildings is also significant – they accounted for $30\%$ of the global final energy consumption in 2021, with electricity constituting about $35\%$ of this building-related energy use \cite{IEA2022}. As a result, there is an urgent need to alleviate the energy burden and maintain a balanced power supply, beginning at the local level, especially within residential buildings.

In response to this need and also due to the growing urgency to reduce carbon dioxide ($\text{CO}_2$) emissions, the concept of net-zero energy houses (ZEHs) \cite{efficiency2015definition} is receiving increasing attention \cite{chen2021carbon}.
A ZEH is a house whose annual net energy consumption is around zero. To achieve this objective, householders are required to actively generate their own energy, primarily through renewable resources. 
Consequently, a householder must satisfy the following energy balance equation to attain the ZEH goal
\begin{equation}
    \sum_{k=0}^{K} G_{k} - \sum_{k=0}^{K} X_{k} = 0,
    \label{eq:ZEH}
\end{equation}
where $G_{k}$ is the energy generation by renewable sources at time instant $k$, $X_{k}$ is the demand at time instant $k$ and $K$ is a certain long time horizon, typically a year. However, satisfying equation \eqref{eq:ZEH} is generally
impossible due to uncertain demand and uncontrollable energy generation.
One alternative is to relax equation \eqref{eq:ZEH} into an inequality constraint. While this might mean that net-zero energy consumption is not achieved precisely, it does ensure that carbon neutrality is met since there is no dependency on non-renewable energy sources. Thus, adding this constraint to the investment optimization problem would help us succeed in the achievement of the ZEH goal.

Among various renewable sources, photovoltaic (PV) technology, in particular, emerges as an appealing option due to its affordable implementation costs and broad adaptability \cite{jager2019pv}.
However, a widely acknowledged limitation of PV systems is their dependency on uncontrollable 
factors like solar radiation, making it difficult to ensure consistent energy generation. 
This becomes even clearer if we compare the common load profiles and PV generation profiles together.
Let us consider, for example, the consumption and PV generation profiles for a typical day in a household in Kitakyushu, Japan, as shown in Fig.~\ref{fig:example_PV_consumption}.
\begin{figure}[tbh]
    \centering
    \includegraphics[width=0.48\textwidth]{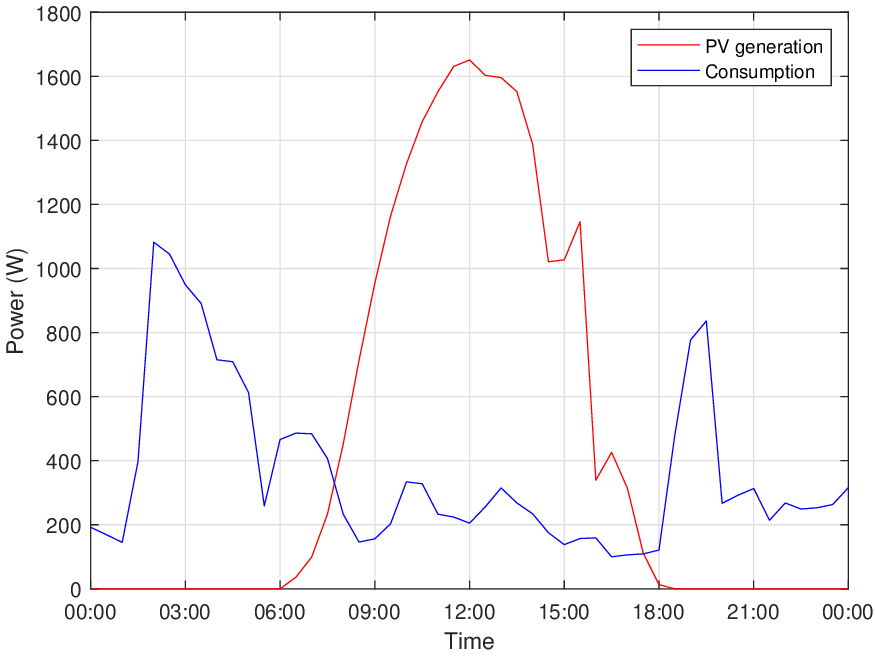}
    \caption{Example of PV and consumption profile for a single day.}
    \label{fig:example_PV_consumption}
\end{figure}
It can be observed that the majority of the energy is generated during hours when there is minimal usage while, conversely, the peak demand typically occurs during periods when no energy generation is taking place. For this reason, it is desirable that self-supply systems are accompanied by a battery, allowing the user to save current unused energy for the near future \cite{hoppmann2014economic}.
It is also important to note that the philosophy of ZEH is not merely about maintaining a balance between energy consumption and production, but also about the optimal use of energy through various strategies, including improved heat insulation, the use of high-efficiency equipment, etc.

The research on optimal sizing for renewable energy and battery systems is vast and encompasses a variety of scenarios. For example, \cite{kumar2023adaptive} considers the investment problem of residential rooftop PV panels in India by studying different market models. However, the integration of a battery system (which might be useful when the generated energy is not enough) is not considered. A similar case study in Philippines can be found in \cite{guno2021optimal}. In addition to PV panels, battery systems have also been included
in the optimization problems appearing in the literature. For instance, \cite{zhu2014optimal} addresses the daily energy flow control problem for the optimal management of residential PV and battery systems, developing a design methodology from the proposed control policy. Similarly, \cite{akram2017innovative} proposes an investment optimization problem considering wind turbines, PV panels, battery systems, and super-capacitors. The problem itself is very complex and cannot be solved by conventional methods. Additional studies addressing these problems from a different perspective are also available, such as \cite{heidari2023planning} where the problem of the loss of inertia in the generation power system due to the introduction of many renewable energy sources is addressed. There, from a regional perspective, an investment problem involving an inertia constraint is solved and compared with the unconstrained problem, quantifying the cost of the introduction of the inertia within the renewable energy-based system.
For an exhaustive list of works, we refer the interested reader to the following review papers \cite{yang2018battery,lian2019review,ridha2021multi}.

However, a significant portion of these studies mainly targets investments in PV panels and battery systems with the primary objective of maintaining the power balance but they often neglect broader aspects like carbon neutrality. Notably, in recent years more research attention has been focused on attaining the ZEH status and decarbonization, as discussed in\cite{yu2016ga,hannan2020review,ahmed2022assessment}.
Yet, a lingering question remains: is achieving net-zero energy for residential houses in specific regions a feasible reality? 
The feasibility of such an objective will be explored in this paper through optimization frameworks.

In the realm of battery-based power system planning and management, there is a widely accepted strategy of purchasing electricity during off-peak hours when it is cheaper and then utilizing it during peak times when prices are higher \cite{ru2012storage, khalilpour2016planning, khezri2022optimal}.
Much like the strategic purchase of electricity during off-peak times, the concept of sharing economy \cite{hu2019sharing,parag2016electricity,sundararajan2017sharing, henni2021sharing} has recently gained increased attention due to the fact that resources are becoming scarcer and more expensive than before. Owners are also trying to monetize their unused assets
(see, for example, the increasing trend of the second-hand apparel market in \cite{kim2022big}). In \cite{alleyne2023control}, sharing economy is defined as ``a disruptive paradigm based on replacing traditional notions of ownership with mechanisms based on sharing/on-demand access''.
Of course, this leads to a profound shift in the traditional notion of ownership. Although there are lots of benefits, it also raises numerous challenges \cite{crisostomi2020analytics}. In this kind of scheme, game theory \cite{marden2015game, chakraborty2018sharing} and multi-agent systems \cite{dorri2018multi,wong2019review} present themselves as the most suitable mechanisms to deal with potential cooperative and adversarial scenarios.

In this paper, we introduce optimization frameworks to determine the optimal size of PV panels and battery systems. Our objective is twofold: to achieve cost savings on future electricity bills and to realize the desired ZEH status across various scenarios. These scenarios encompass both an individual's self-supply approach and a collaborative investment strategy involving multiple users.
Although the problems we initially present are nonlinear, subsequent transformations allow us to relax them into linear programming (LP) problems, which are tractable. Additionally, our numerical examples leverage historical solar generation and household consumption data, rather than relying on synthetic data, enhancing the authenticity of our findings.
It is worth noting that we explored the optimal sizing of PV panels and batteries using optimization and game theory in a stochastic context in our earlier paper \cite{li2023stochastic}. A notable distinction between the two studies is the consideration of battery dynamics. While our prior research did not account for this aspect, the accumulation of power resulting from battery dynamics emerges as a pivotal factor for achieving net-zero energy in this paper.

\subsection{Contributions}
The contributions of this paper are listed as follows.
\begin{enumerate}
    \item We formulate a novel optimization problem to determine the optimal sizing of PV panels and batteries from an economic perspective while considering the achievement of the ZEH status. 
    \item In the proposed optimization problem, both the cases where users invest individually and a sharing-economy approach are considered, showing quantitatively the advantages that the latter provides in a cost-wise manner.
    \item By employing relaxation techniques, we transform the original nonconvex problem into an LP problem with equivalent solutions. Importantly,  the attainment of the ZEH status in our formulation can be realized by simply adding a constraint to the LP.
    \item We validate the effectiveness of our proposed formulation using real-world data sourced from Kitakyushu, Japan. The results illustrate that the inclusion of the ZEH constraint causes only a minimal increase in the optimal value of the LP. Thus, through this practical data, we demonstrate that attaining ZEH status does not significantly elevate costs, which implies that it is a feasible goal for residential houses in Japan.
\end{enumerate}

\subsection{Structure}
The paper is structured as follows: in Section \ref{sec:single}, the basic problem for complete self-supply is presented. Then, in Section \ref{sec:LP}, some new decision variables are introduced in such a way that the new optimization problem can be posed as an LP. Also, several properties of the resulting model are presented. Section \ref{sec:examples} shows the numerical examples of the proposed optimization problems and some discussions about the obtained results. Finally, in Section \ref{sec:conclusions}, the conclusions of the paper and the expected future work are presented.

\section{Optimal Sizing of PV Panels and Battery}
\label{sec:single}
In this section, we consider the optimal investment of PV panels and battery for an individual household assuming the availability of past solar generation and demand data.

The primary objective is to maintain power balance. When the electricity generated by PV panels exceeds the consumption and battery capacity, the surplus energy is returned to the grid.
Conversely, when the generated energy falls short, the battery discharges to meet the demand. When the battery reaches the point of depletion, a fuel cell is activated, generating the required energy through fuel combustion.

It is important to recognize the challenges this power balance issue presents. On one hand, the reverse power might lead to overvoltage issues in the feeders, potentially damaging the power grid
\cite{tonkoski2012impact,katiraei2011solar,baran2012accommodating}, which we would like to minimize when determining the optimal size for PV panels and the battery (note that this might be particularly important when dealing with microgrids \cite{akbari2022optimized}).
On the other hand, to promote the integration of renewable energies into the electrical system, feed-in tariffs (FiT) are commonly utilized. They provide an incentive to users by offering financial compensation for surplus power.
Both policies are taken into account, respectively, in our following formulation.

Given the factors previously stated, we should tackle in an economic manner (i.e. money loss or investment) the following terms in our optimization problem:
\begin{itemize}
    \item Investment in PV panels.
    \item Investment in the battery system.
    \item Reverse power penalty or FiT profit.
    \item Fuel cell generation cost.
\end{itemize}
First, we present in the next subsection the model of the battery dynamics, which is needed to compute the amount of reverse power penalty, FiT profit, and fuel cell generation cost.

\subsection{Battery dynamics}
Here, we consider the state-of-charge (SoC) dynamics as a bounded integrator, that is
\begin{equation}
\label{eq:first_model}
C_{k} =   \left[ C_k^{+} \right]^{\bar{C}}_{0}   :=
    \begin{cases}
        \underline{\alpha} \bar{C} & \text{if } C_{k}^{+} < \underline{\alpha} \bar{C} \\
        \overline{\alpha} \bar{C} & \text{if }  C_{k}^{+}  > \overline{\alpha} \bar{C} \\
        C_{k}^{+} & \text{else} 
    \end{cases}
\end{equation}
where $C_{k}$ is the SoC of the battery at time $k$, $\bar{C}$ is the maximum capacity of the battery, $\overline{\alpha} \in (0.5,1]$ and $\underline{\alpha} \in [0,0.5)$ are scalars defining the charging cycle of the battery and $C_{k}^{+}$ is the SoC of the battery at time $k$ before applying the bounds, i.e.
\begin{equation}
    C_{k}^{+} = \gamma \,C_{k-1} + a Y_{k-1} - X_{k-1},  
    \label{eq:battery_model}
\end{equation}
where $\gamma \in (0,1]$ is a scalar representing the losses (self-discharge, etc.) from one instant to another, $a$ is the total area of the PV panels, $Y_{k-1}$ is the energy production of the photovoltaic panels per $\text{m}^{2}$, i.e., $\text{kWh}/\text{m}^{2}$ at time instant $k-1$ and $X_{k-1}$ is the energy consumption in $\text{kWh}$ at time instant $k-1$.
It is easy to see that $C_{k}$ takes the value of $C_{k}^{+}$ when it is within the operating range of the battery and it is saturated when it is outside the range. For a discussion regarding the validity of the model, we refer the interested reader to the appendix.

Also, it is possible to consider that the charging rate is limited. This constraint limits the speed of the charging and discharging process
\begin{equation}
       |C_{k}-C_{k-1}| \leq R \,\bar{C} , 
\end{equation}
where $R \in (0,1]$. Note that this constraint can be written as two linear constraints, i.e.
\begin{equation}
    C_{k}-C_{k-1} \leq R \,\bar{C}, \quad
    C_{k-1}-C_{k} \leq R \,\bar{C}.
\end{equation}
From now on and for the sake of clarity, we assume that constraints referred to a certain time instant $k$ must always be fulfilled $\forall k=1,\dots, K$, and thus it will be omitted in the following optimization problems.
\subsection{Cost function}
Once the dynamics of the battery has been presented, the terms of the cost function can be shown:
\begin{itemize}
    \item Investment cost of the PV panels, which is quantified by $\Pi_{PV} a$. Note that $\Pi_{PV}$ is a constant value indicating the price per square meter $[\text{m}^{2}]$ and $a$ is a decision variable showing how much area $[\text{m}^{2}]$ should be bought. 
    \item Investment cost of the battery, which is quantified by $\Pi_{B} \bar{C}$. Here, the price is assumed to follow a linear function although other possibilities can be considered without loss of generality. Note that $\bar{C}$ is a decision variable indicating the size of the battery $[\text{kWh}]$ and  $\Pi_{B}$ corresponds to the price per $\text{kWh}$.
    \item Reverse power penalty or FiT profit, which is quantified by $\Pi_{R} \max ( C_{k}^{+}-\overline{\alpha}\bar{C},0)$. Here, 
    $\Pi_{R}$ is the cost/profit for each kWh injected into the main grid and $\max (\cdot, \cdot)$ is a function $\mathbb{R} \times \mathbb{R} \to \mathbb{R}$ that returns the maximum value of the two inputs scalars. 
    This cost/profit is equal to zero when the generated energy can be saved in the battery and $\Pi_{R} ( C^{+}_{k}-\overline{\alpha}\bar{C})$ elsewhere. Note that making $\Pi_{R}<0$ tackles the case of the FiT where the user receives compensation for his/her surplus energy.
    \item Fuel cell generation cost, which is quantified by $\Pi_{G} \max (\underline{\alpha}\bar{C} -C_{k}^{+},0)$. This cost appears when the PV generation and the stored energy in the battery is not enough to satisfy the demand, that is, its value is zero whenever $C_{k}^{+} \geq \underline{\alpha}\bar{C}$ and greater than zero when $C_{k}^{+} < \underline{\alpha}\bar{C}$. 
    In this paper, we assume that this additional energy is generated by means of a fuel cell and thus $\Pi_{G}$ represents the cost of generating each kWh in this fuel cell.
\end{itemize}

\subsection{Problem formulation}
In summary, we formulate the following optimization problem to determine the PV size $a$ and the capacity $\bar C$:
\begin{subequations}
\label{eq:original_after_remark}
\begin{align}
     \min_{a,\bar{C},\textbf{C},\textbf{C}^{+}} \quad & \Pi_{PV} a + \Pi_{B} \bar{C} + \sum_{k=1}^{K} \Pi_{R} \max ( C_{k}^{+}-\overline{\alpha} \bar{C},0) \nonumber \\&+ \sum_{k=1}^{K} \Pi_{G} \max (\underline{\alpha}\bar{C}-C_{k}^{+},0) \\
    \text{s.t.} \quad & C_{k} =  \left[ C_k^{+} \right]^{\overline{\alpha} \bar{C}}_{\underline{\alpha} \bar{C}}   
    \\
    & C_{k}^{+} = \gamma \,C_{k-1} + a Y_{k-1} - X_{k-1}  \\
    & C_{k}-C_{k-1} \leq R \,\bar{C} \label{eq:constraint_R_1} \\
    & C_{k-1}-C_{k} \leq R \,\bar{C} \label{eq:constraint_R_2} \\
    & 0 \leq a \leq a_{\max}  \label{eq:constraint_area}
    \\
    & C_{0} = \hat{C} \label{eq:constraint_initial}  \\
    &  a \sum_{k=0}^{K} Y_{k} \geq \sum_{k=0}^{K} X_{k}  \label{eq:ZEH_constraint}
\end{align}
\end{subequations}
where $K$ is the time horizon of the optimization problem and \textbf{C} and $\textbf{C}^+$ are the vectors of concatenated $C_k$ and $C_k^+$, that is
\begin{equation}
\textbf{C} = \{C_k\}_{k=1}^K, \quad \textbf{C}^+ = \{C_k^+\}_{k=1}^K.
\end{equation}
As we must take into account that the amount of PV panels that can be installed on the rooftop of each house is limited, we consider that a maximum of $a_{\max} \; \text{m}^2$ of panels can be bought, which leads to constraint \eqref{eq:constraint_area}. On the other hand, constraint \eqref{eq:constraint_initial} establishes the initial SoC of the battery to a certain scalar $\hat{C} \in [\underline{\alpha} \bar{C},\overline{\alpha}\bar{C}]$.
Finally, constraint \eqref{eq:ZEH_constraint}
forces users to achieve ZEH for the time period given by $K$. 
Note that the nonconvex saturation in \eqref{eq:first_model} and the $\max$ function make this optimization problem hard to solve in general. 

There is a slightly different variation of the previous formulation where users share the battery and the PV panels. Assuming the existence of $N$ users, the optimization problem would be as follows:
\begin{subequations}
\label{eq:original_afer_remark_community}
\begin{align}
     \min_{a_i,\bar{C},\textbf{C},\textbf{C}^{+}} \;\; & \Pi_{B} \bar{C} + \sum_{i=1}^{N} \Pi_{PV} a_{i}  + \sum_{k=1}^{K} \Pi_{R} \max ( C_{k}^{+}-\overline{\alpha}\bar{C},0) \nonumber \\&+ \sum_{k=1}^{K} \Pi_{G} \max (\underline{\alpha}\bar{C} -C_{k}^{+},0) \\
    \text{s.t.} \quad & C_{k} =  \left[ C_k^{+} \right]^{\overline{\alpha} \bar{C}}_{\underline{\alpha} \bar{C}}   
    \\
    & C_{k}^{+} = \gamma \,C_{k-1} + \sum_{i=1}^{N} \left( a_{i}  Y_{k-1,i} - X_{k-1,i} \right) \label{eq:power_balance_community} \\
    &  0 \leq a_{i} \leq a_{\max}  \quad \forall i=1,\hdots,N  \label{eq:ai_community} \\
        & \text{equations (\ref{eq:constraint_R_1}, \ref{eq:constraint_R_2}$,$\ref{eq:constraint_initial})} \label{eq:bunch} \\
        &  \sum_{i=1}^{N} a_{i} \sum_{k=0}^{K} Y_{k,i} \geq \sum_{i=1}^{N} \sum_{k=0}^{K} X_{k,i}  \label{eq:ZEH_community}
\end{align} 
\end{subequations}
where $a_{i}$ is the amount of PV panels for householder $i$ and $Y_{k,i}$ and $X_{k,i}$ are the generation and consumption at time instant $k$ for user $i$ respectively.
Note that, in this formulation,  we consider the group of users as a unique ``big user'' since the generation and consumption of each user is summed in the power balance equation \eqref{eq:power_balance_community}. 
\section{Relaxation to Linear Programming}
\label{sec:LP}
In what follows, we present how 
problems \eqref{eq:original_after_remark} and \eqref{eq:original_afer_remark_community} can be relaxed to LPs which can be easily solved. 

Firstly, we tackle the problem of removing the saturation constraints. For this purpose, we define new decision variables $\Phi \in \mathbb{R}^{K}$. Now, denote $\phi_{k}$ as the $k$-th element of $\Phi$. Then, these decision variables $\phi_{k}$ can be used to ``absorb'' the difference between $C_{k}^{+}$ and $C_{k}$. 
Thanks to the definition of these new decision variables $\phi_{k}$, the saturation constraints
\eqref{eq:first_model} become linear constraints, i.e.
\begin{equation}
\begin{array}{c}
     C_{k} + \phi_{k} = \gamma \,C_{k-1} + a Y_{k-1} - X_{k-1}, \\[\smallskipamount]
     \underline{\alpha} \bar{C} \leq C_{k} \leq \overline{\alpha} \bar{C}.
\end{array} 
\end{equation}
That is, when the battery is full, $\phi_{k}$ takes a positive value corresponding to the amount of energy that cannot be stored. On the other hand, when the battery becomes empty, $\phi_{k}$ takes a negative value corresponding to the amount of energy that must be generated using the fuel cell to avoid a blackout. Thus, it is clear that 
the two terms below are associated with
the cost of ``reverse power'' (or FiT) and ``fuel cell'', respectively:
\begin{equation}
\begin{array}{c}
   \Pi_{R} \max ( C_{k}^{+}- \overline{\alpha}\bar{C},0) \rightarrow \Pi_{R} \max ( \phi_{k},0), \\ [\smallskipamount]
   \Pi_{G} \max (\underline{\alpha}\bar{C} -C_{k}^{+},0) \rightarrow \Pi_{G} \max (- \phi_{k},0).
\end{array}  
\end{equation}
Under the reasonable assumption that $\Pi_{G} \geq 0$ and $\Pi_{R} \geq -\Pi_{G}$, then $\Pi_{R} \max ( \phi_{k},0) + \Pi_{G} \max (- \phi_{k},0)$ is a convex function and thus we have the following equivalence
\begin{equation}
\begin{array}{c}
\Pi_{R} \max ( \phi_{k},0) + \Pi_{G} \max (- \phi_{k},0) = \\[\smallskipamount]
\max (\Pi_{R} \phi_{k},- \Pi_{G}  \phi_{k}).
\end{array}  
    \label{eq:new_max_cost_function}
\end{equation}
Then, by substituting equation \eqref{eq:new_max_cost_function} in the cost function and taking into account the previous change in the constraints, the new optimization problem is
\begin{subequations}
    \begin{align}
\min_{a,\bar{C},\textbf{C},\Phi^{+},\Phi^{-}}  \hspace{-1mm} & \Pi_{PV} a + \Pi_{B} \bar{C} + \sum_{k=1}^{K} \max (\Pi_{R} \phi_{k},- \Pi_{G}  \phi_{k}) \\
\text{s.t.} \quad 
& 
C_{k} + \phi_{k} = \gamma \, C_{k-1} + a Y_{k-1} - X_{k-1}  \\
& \underline{\alpha} \bar{C} \leq C_{k} \leq \overline{\alpha} \bar{C}  \\
        & \text{equations (\ref{eq:constraint_R_1}$\sim$\ref{eq:ZEH_constraint})}
\end{align} 
\end{subequations}
where $\Phi$ is the vector of concatenated $\phi_k$, that is
$
\Phi = \{\phi_k\}_{k=1}^K.
$ 
Note that all constraints are linear but the cost function remains nonlinear. As this $\max$ function corresponds to a convex piecewise linear function, it is possible to transform this problem into an LP problem which can be solved efficiently. 
This can be done by adding some slack variables (see \cite[\S 4.3.1]{Boyd04}). 
Here, for the sake of clarity, we opt to define  $\phi_{k}$ as the difference between two nonnegative variables 
\begin{equation}
\phi_{k} = \phi_{k}^{+} - \phi_{k}^{-}, 
\end{equation}
where $\phi_{k}^{+} \geq 0$ and $\phi_{k}^{-} \geq 0$. Then, $\phi_{k}^{+}$ would correspond to the amount of reverse power and $\phi_{k}^{-}$ would correspond to the amount of energy generated by means of the fuel cell at time $k$. This is similar to the strategies used to pose an LP problem in the standard form (see \cite[\S 4.3]{Boyd04}). Note that, after this change, we have that
\begin{equation}
\max (\Pi_{R} \phi_{k},- \Pi_{G}  \phi_{k}) = \Pi_{R} \phi_{k}^{+} + \Pi_{G}  \phi_{k}^{-}
\end{equation}
and thus the problem can be written as
\begin{subequations}
\label{eq:num_example_individual}
\begin{align}
\min_{a,\bar{C},\textbf{C},\Phi^{+},\Phi^{-}}  \hspace{-1mm} & \Pi_{PV} a + \Pi_{B} \bar{C} + \sum_{k=1}^{K} \left( \Pi_{R} \phi_{k}^{+} + \Pi_{G}  \phi_{k}^{-} \right) \\
\text{s.t.} \quad & 
C_{k} + \phi_{k}^{+} - \phi_{k}^{-} = \gamma C_{k-1} + a Y_{k-1} - X_{k-1}  \\
& \underline{\alpha} \bar{C} \leq C_{k} \leq \overline{\alpha} \bar{C}  \\
& \phi_{k}^{+} \geq 0, \quad \phi_{k}^{-} \geq 0 \\
        & \text{equations (\ref{eq:constraint_R_1}$\sim$\ref{eq:ZEH_constraint})}
\end{align} 
\end{subequations}
where $\Phi^+$ and $\Phi^-$ are the vectors of concatenated $\phi_k^+$ and $\phi_k^-$, that is
\begin{equation}
\Phi^+ = \{\phi^+_k\}_{k=1}^K, \quad \Phi^- = \{\phi_k^-\}_{k=1}^K.
\end{equation}
Similarly, the community-based optimization problem can be easily obtained from the aforementioned individual optimization problem by considering the group of users as a unique ``big user''. Then, it is possible to write the following optimization problem:
\begin{subequations}
\label{eq:num_example_community}
\begin{align}
\min_{a_i,\bar{C},\textbf{C},\Phi^{+},\Phi^{-}}  \hspace{-3mm} & \Pi_{B} \bar{C} + \sum_{i=1}^{N} \Pi_{PV} a_i + \sum_{k=1}^{K} \left( \Pi_{R} \phi_{k}^{+} + \Pi_{G}  \phi_{k}^{-} \right) \\
\text{s.t.} \quad  
& 
C_{k} + \phi_{k}^{+} - \phi_{k}^{-} = \nonumber \\ & \qquad \qquad \gamma \, C_{k-1} + \sum_{i=1}^{N} \left( a_i Y_{k-1,i} - X_{k-1,i} \right) \\
& \underline{\alpha}\bar{C} \leq C_{k} \leq \overline{\alpha}\bar{C} \\
        & \text{equations (\ref{eq:ai_community}$\sim$\ref{eq:ZEH_community})}. 
\end{align} 
\end{subequations}

\subsection{Properties of the proposed model}
This subsection is dedicated to show the equivalence of the proposed model \eqref{eq:num_example_individual} with respect to the original model including the saturation constraints \eqref{eq:original_after_remark} and the conditions for this equivalence to hold. First, we show that the optimal cost of the original problem is an upper bound of the optimal cost of the proposed LP problem.
\begin{theorem}[Upper bound of the optimal cost]
The optimal cost of the LP problem \eqref{eq:num_example_individual} is less than or equal to the optimal cost of the original problem \eqref{eq:original_after_remark}.
\end{theorem}
\begin{proof}
It is possible to rewrite the original optimization problem \eqref{eq:original_after_remark} by making the following change of variables 
\begin{equation}
C_{k}^{+} = C_{k} + \phi_{k},
\end{equation}
where $\phi_{k}  \in \mathbb{R}$.
Then, the resulting optimization problem is
\begin{subequations}
    \begin{align}
     \min_{a,\bar{C},\textbf{C},\Phi} \; & \Pi_{PV} a + \Pi_{B} \bar{C} + \sum_{k=1}^{K} \Pi_{R} \max ( C_{k} + \phi_{k} - \overline{\alpha} \bar{C},0) \nonumber\\&+ \sum_{k=1}^{K} \Pi_{G} \max (\underline{\alpha}\bar{C} - C_{k} - \phi_{k},0) \\
    \text{s.t.} \quad & C_{k} =  \left[ C_{k} + \phi_{k} \right]^{\overline{{\alpha}}\bar{C}}_{\underline{\alpha}\bar{C}}   \\
    & C_{k} + \phi_{k} = \gamma \,C_{k-1} + a Y_{k-1} - X_{k-1}
    \\
& \text{equations (\ref{eq:constraint_R_1}$\sim$\ref{eq:ZEH_constraint})}.
\end{align}
\end{subequations}
Also, note that
\begin{equation}
\begin{array}{c}
    \max ( C_{k} + \phi_{k} - \overline{\alpha}\bar{C},0) =  \max (\phi_{k},0),  \\[\smallskipamount] \max (\underline{\alpha}\bar{C}-C_{k} - \phi_{k},0) = \max (- \phi_{k},0),
\end{array}  
\end{equation}
because of the constraint $C_{k} =  \left[ C_{k} + \phi_{k} \right]^{\overline{\alpha}\bar{C}}_{\underline{\alpha}\bar{C}}$. After making $\phi_{k} = \phi_{k}^{+}-\phi_{k}^{-}$, where $\phi_{k}^{+},\phi_{k}^{-}\geq 0$, the problem becomes
\begin{subequations}
    \begin{align}
     \min_{a,\bar{C},\textbf{C},\Phi^{+},\Phi^{-}} \; & \Pi_{PV} a + \Pi_{B} \bar{C} + \sum_{k=1}^{K} \Pi_{R} \max ( C_{k} + \phi_{k} - \overline{\alpha} \bar{C},0) \nonumber \\&+ \sum_{k=1}^{K} \Pi_{G} \max (\underline{\alpha}\bar{C} - C_{k} - \phi_{k},0) \\
    \text{s.t.} \quad & C_{k} =  \left[ C_{k} + \phi_{k}^{+} - \phi_{k}^{-} \right]^{\overline{{\alpha}}\bar{C}}_{\underline{\alpha}\bar{C}}   
    \\
    & C_{k} + \phi_{k}^{+} - \phi_{k}^{-} = \gamma \,C_{k-1} + a Y_{k-1} - X_{k-1} \\
    & \text{equations (\ref{eq:constraint_R_1}$\sim$\ref{eq:ZEH_constraint})}.
\end{align}
\end{subequations}
Finally, it is easy to see that the constraint $C_{k} =  \left[ C_{k} + \phi_{k}^{+} - \phi_{k}^{-} \right]^{\overline{\alpha}\bar{C}}_{\underline{\alpha}\bar{C}}$ implies $\underline{\alpha}\bar{C} \leq C_{k} \leq \overline{\alpha}\bar{C}$ too and thus we can add the second constraint to the optimization problem without changing the solution nor the optimal cost, that is
\begin{subequations}
    \begin{align}
\min_{a,\bar{C},\textbf{C},\Phi^{+},\Phi^{-}} \; & \Pi_{PV} a + \Pi_{B} \bar{C} + \sum_{k=1}^{K} \left( \Pi_{R} \phi_{k}^{+} +  \Pi_{G} \phi_{k}^{-} \right) \\
    \text{s.t.} \quad & C_{k} =  \left[ C_{k} + \phi_{k}^{+} - \phi_{k}^{-} \right]^{\overline{\alpha}\bar{C}}_{\underline{\alpha}\bar{C}}   
    \\
    & C_{k} + \phi_{k}^{+} - \phi_{k}^{-} = \gamma \, C_{k-1} + a Y_{k-1} - X_{k-1}   \\
    & \underline{\alpha}\bar{C} \leq C_{k} \leq \overline{\alpha}\bar{C}
    \\ & \phi_{k}^{+} \geq 0, \quad \phi_{k}^{-} \geq 0 \\
    & \text{equations (\ref{eq:constraint_R_1}$\sim$\ref{eq:ZEH_constraint})}.
\end{align}
\end{subequations}
At this point, it is clear that the proposed optimization problem in \eqref{eq:num_example_individual} corresponds to a less constrained version of the original optimization problem in \eqref{eq:original_after_remark} because the only difference is the absence of the constraint $C_{k} =  \left[ C_{k} + \phi_{k}^{+} - \phi_{k}^{-} \right]^{\overline{\alpha}\bar{C}}_{\underline{\alpha}\bar{C}}$ in the proposed LP  problem. Thus, the optimal cost of the proposed optimization problem will be less or equal to the cost of the original optimization problem. 
\end{proof}

In the following, it is shown that the additional degrees of freedom obtained because of the removal of the saturation constraints do not provide a better cost when $\Pi_{R} > 0$ and $\Pi_{G} > 0$ and thus the equality holds. In the next Lemma, we show that the optimal solution is not unique under certain conditions.

\begin{lemma}[Non-uniqueness of the optimal solution]
\label{lemma2}
Assume that $\Delta \phi^{+}$ or $\Delta \phi^{-}$ exists so that $\phi_{j}^{+} = \phi_{j}^{+^{*}} + \Delta \phi^{+}$, $\phi_{j-1}^{+} = \phi_{j-1}^{+^{*}} - \Delta \phi^{+}$ and $\phi_{j}^{-} = \phi_{j}^{-^{*}} + \Delta \phi^{-}$, $\phi_{j-1}^{-} = \phi_{j-1}^{-^{*}} - \Delta \phi^{-}$ for a certain $j$ while fulfilling the constraints $\underline{\alpha}\bar{C} \leq C_{k} \leq \overline{\alpha}\bar{C}$ and $\phi_{k}^{+} \geq 0$ or $\phi_{k}^{-} \geq 0$ for every $k$. Then, there exists an infinite number of optimal solutions for the optimization problem \eqref{eq:num_example_individual}.
\end{lemma}
\begin{proof}
For simplicity, we tackle in what follows the case of $\Delta \phi^{+}$ as a similar reasoning can be done for $\Delta \phi^{-}$.
Denote $a^{*}$, $\bar{C}^{*}$, $\textbf{C}^{*}$, $\Phi^{+^{*}}$ and $\Phi^{-^{*}}$ an optimal solution of the optimization problem \eqref{eq:num_example_individual}. Assume the existence of a certain $\Delta \phi^{+}$ and $j$ so that 
\begin{equation}
\label{eq:vector_phi:vector_C}
\begin{array}{c}
        \Phi^{+}_{\Delta} = [
      \phi^{+^{*}}_{1},     \dots,  \phi^{+^{*}}_{j-1} - \Delta \phi^{+} , \phi^{+^{*}}_{j} + \Delta \phi^{+} ,         \dots , \phi^{+^{*}}_{K}]^{\top}, \\[\smallskipamount]
       \textbf{C}_{\Delta} = [C_{0}^{*}, \dots, C_{j-1}^{*} + \Delta \phi^{+}, C_{j}^{*}, \dots , C_{K}^{*} ]^{\top}   ,
\end{array}  
\end{equation}
and the constraints $\underline{\alpha}\bar{C} \leq C_{k} \leq \overline{\alpha}\bar{C}$ and $\phi_{k}^{+} \geq 0$ are still satisfied for every $k$. Note that the SoC of the battery only changes at $j-1$ since 
\begin{equation}
   C_{j-1} + \phi_{j-1}^{+^{*}} - \Delta \phi^{+} -\phi_{j-1}^{-^{*}}  = C_{j-2}^{*} + a^{*} Y_{j-2} - X_{j-2} . 
\end{equation}
Note that 
\begin{equation}
    C_{j-1}^{*} = C_{j-2}^{*} + a^{*} Y_{j-2} - X_{j-2} -\phi_{j-1}^{+^{*}} + \phi_{j-1}^{-^{*}}
\end{equation}
and thus
\begin{equation}
C_{j-1} = C_{j-1}^{*} + \Delta \phi^{+}.
\end{equation}
Once the value of $C_{j-1}$ is known, we can compute the value of $C_{j}$
\begin{gather}
    C_{j} + \phi_{j}^{+} - \phi_{j}^{-^{*}}  = C_{j-1} + a^{*} Y_{j-1} - X_{j-1} \nonumber
\\
C_{j} + \phi_{j}^{+^{*}} +\Delta \phi^{+} - \phi_{j}^{-^{*}}  = C_{j-1}^{*} + \Delta \phi^{+} + a^{*} Y_{j-1} - X_{j-1} \nonumber  \\
C_{j} + \phi_{j}^{+^{*}} - \phi_{j}^{-^{*}}  = C_{j-1}^{*} + a^{*} Y_{j-1} - X_{j-1} .
\end{gather}
Hence, $C_{j} = C_{j}^{*}.$
Denoting sum$(\cdot)$ as the sum of the elements of a vector, then the new cost can be compared with respect to the optimal cost:
\begin{align}
    \Delta V =& \Pi_{PV} a^{*} + \Pi_{B} \bar{C}^{*} +  \Pi_{R} \text{sum}\left(\Phi^{+}_{\Delta}\right) + \sum_{k=1}^{K} \Pi_{G}  \phi_{k}^{-^{*}} \nonumber \\
    -& \Pi_{PV} a^{*} - \Pi_{B} \bar{C}^{*} - \sum_{k=1}^{K} \Pi_{R} \phi_{k}^{+^{*}} - \Pi_{G}  \phi_{k}^{-^{*}} \nonumber \\
    =& \Pi_{R} (\Delta \phi^{+} - \Delta \phi^{+} ) + \sum_{k=1}^{K} \Pi_{R} \phi_{k}^{+^{*}} - \sum_{k=1}^{K} \Pi_{R} \phi_{k}^{+^{*}} = 0. 
\end{align}
Thus, the variables in \eqref{eq:vector_phi:vector_C} for any feasible $\Delta\phi^+$ together with $a^*$, $\bar{C}^*$, $\Phi^-$ make up the set of optimal solutions of \eqref{eq:num_example_individual}
\end{proof}
The following Theorem follows from the previous results.
\begin{theorem}[Infinite solutions]
\label{theorem2}
Consider the optimization problem in \eqref{eq:num_example_individual}. If one of the below proposition holds:
\begin{itemize}
    \item There is at least one time slot where the battery is not empty, the fuel cell is off, the battery reaches full capacity in the subsequent time slot and the charging inequality constraint is not active (existence of feasible $\Delta \phi^{+}$).
    \item There is at least one time slot where the battery is not full, no reverse power is being returned to the grid, the battery is completely empty at next time instant and the charging inequality constraint is not active (existence of feasible $\Delta \phi^{-}$).
\end{itemize}
Then, it is possible to find an infinite number of optimal solutions.
\end{theorem}
\begin{proof}
The conditions appearing in the first proposition corresponds to the following: $C_{j}^{*} = \overline{\alpha} \bar{C}$, $\underline{\alpha}\bar{C} < C_{j-1}^{*} \leq \overline{\alpha}\bar{C}$, $\phi_{j-1}^{-^{*}}=0$, $C_{k}-C_{k-1} < R \,\bar{C}$ and $C_{k-1}-C_{k} < R \,\bar{C} $. Then $\phi_{j}^{-^{*}}=0$ and $\phi_{j}^{+^{*}}  = C_{j-1}^{*} + a^{*} Y_{j-1} - X_{j-1} - \bar{C} \geq 0$. From these conditions, it is easy to see that it is possible to find a certain $\Delta \phi^{+} < 0$ so that the constraints $\underline{\alpha}\bar{C} \leq C_{k} \leq \overline{\alpha}\bar{C} $ and $\phi_{k}^{+} \geq 0$ are still fulfilled for every $k$. Due to the results of Lemma \ref{lemma2}, this proves that an infinite number of optimal solutions can be found.
On the other hand, the conditions of the second proposition are the following: $C_{j}^{*} = \underline{\alpha}\bar{C}$, $\underline{\alpha}\bar{C} \leq C_{j-1}^{*} < \overline{\alpha}\bar{C}$, $\phi_{j-1}^{+^{*}}=0$, $C_{k}-C_{k-1} < R \,\bar{C}$ and $C_{k-1}-C_{k} < R \,\bar{C}$. Then $\phi_{j}^{+^{*}}=0$ and $\phi_{j}^{-^{*}}  = X_{j-1} - C_{j-1}^{*} - a^{*} Y_{j-1}  \geq 0$. Similarly, from the above conditions, it can be seen that it is possible to find a certain $\Delta \phi^{-} < 0$ so that the constraints $\underline{\alpha}\bar{C} \leq C_{k} \leq \overline{\alpha}\bar{C} $ and $\phi_{k}^{-} \geq 0$ are still fulfilled for every $k$. As in the previous case, for this situation, Lemma \ref{lemma2} states that an infinite number of optimal solutions can be found.
\end{proof}
\begin{corollary}
        By looking at the results of Theorem \ref{theorem2}, it is clear that all possible trajectories of the SoC of the battery which can be obtained by altering $\textbf{C}_{\Delta}$, $\Phi^{+}_{\Delta}$ and $\Phi^{-}_{\Delta}$ attain the same optimal cost when $\Pi_{R} > 0$ and $\Pi_{G} > 0$. In other words,  when $\Pi_{R} > 0$ and $\Pi_{G} > 0$, the optimal cost of \eqref{eq:original_after_remark} and \eqref{eq:num_example_individual} are the same.
\end{corollary}
 On the other hand, when considering FiT (i.e. $\Pi_{R} < 0$), the value of the optimal cost can be smaller than the cost obtained with the original optimization problem. This is due to the fact that $\Phi^{+}$ and $\Phi^{-}$ can be interpreted as control variables allowing us to discharge the battery or generate energy in the fuel cell at any time, not only when it is strictly needed.

\section{Case Study in Kitakyushu, Japan}
\label{sec:examples}
In this section, we conduct simulation studies using real data obtained from Jono, a neighborhood located in Kitakyushu, Japan. The available data corresponds to the time period from April 1, 2021 to February 28, 2022. This data includes both the consumption and solar generation for a total of 134 households every 30 minutes during this period of time (11 months). Different scenarios (each one with its own set of parameters) are presented to show how the proposed optimization problem performs and to discuss the obtained results.

The problem to be solved for the individual case corresponds to the problem in \eqref{eq:num_example_individual}
where the last constraint can be considered or not depending if we want to enforce ZEH.
Similarly, the optimization problem to be solved for the community-based approach corresponds to the optimization problem in \eqref{eq:num_example_community} where the last constraint is optional.

Regarding the value of the parameters, we consider a depth of discharge of $0.9$, that is $\underline{\alpha}=0.05$ and $\overline{\alpha} = 0.95$. The charging rate is limited to the $50\%$ of the total capacity of the battery ($R = 0.5$), i.e. the battery can be fully charged in 2 time steps (1 hour). 
On the other hand, $\gamma$ is chosen to be $\gamma=0.99998$ and the initial condition is fixed to $C_{0} = \underline{\alpha} \bar{C}$.
Finally, $\Pi_{PV} = 1000 \; \yen / \text{m}^2$, $\Pi_{B} = 4500 \; \yen / \text{kWh}$, $\Pi_{G} = 30 \; \yen /\text{kWh}$.  Note that $\Pi_{PV}$ and $\Pi_{B}$ are chosen so that the cost of the battery and the solar panels is amortized over a span of 10 years, which is the estimated life of $\text{LiFePO}_{4}$ batteries of these characteristics according to most manufacturers. As $\Pi_{R}$ changes in every proposed scenario, we present its value at the beginning of each subsection.

\subsection{Scenario 1. Penalized reverse power}
Here, we consider that the energy that cannot be stored in the battery penalizes the cost function of the optimization problem, that is $\Pi_{R} = 10 \; \yen /\text{kWh}$.

The results are shown in Fig.~\ref{fig:aC_4500_10}, \ref{fig:histogram_4500_10},  \ref{fig:cost_4500_10} and Table~\ref{table:4500_rev10}. Fig.~\ref{fig:aC_4500_10} shows the amount of PV panels in $\text{m}^{2}$ and battery size in $\text{kWh}$ for each user. Blue dots correspond to the problem without the ZEH constraint whereas red dots correspond to the problem including the ZEH constraint. It is clear that, due to the term penalizing the reverse flow, the investment in PV panels and batteries can be considered small in comparison to the problem including the ZEH constraint. In particular, for the individual case, the increase in investment for the PV panels is 34.02\% on average and 24.15\% on average for the investment on the battery system. For the community-based approach, we obtain that the investment in PV panels increases by 27.53\% and the investment in batteries increases by 12.49\%. Similar results can be seen in Fig.~\ref{fig:histogram_4500_10}, where histograms of the investment of PV panels (on the top) and battery (on the bottom) are shown. 
In addition, average investments of the individual and community approaches are presented. It can be observed that allowing energy sharing within the community incentivizes investment in batteries, while simultaneously decreasing the average need for individual PV panels. This leads to a more energy-efficient management by reducing waste from over-investment.
\begin{figure}[tbh]
    \centering
    \includegraphics[width=0.48\textwidth]{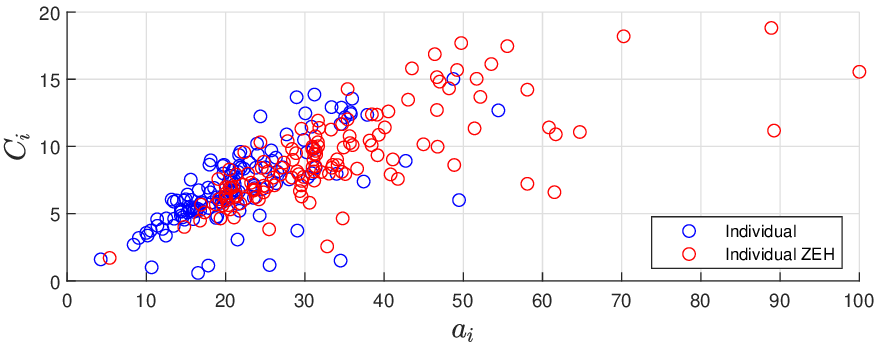}
    \caption{PV panels and batteries investments with and without ZEH constraints for the penalized reverse power scenario}. 
    \label{fig:aC_4500_10}
\end{figure}
\begin{figure}[tbh]
    \centering
    \includegraphics[width=0.48\textwidth]{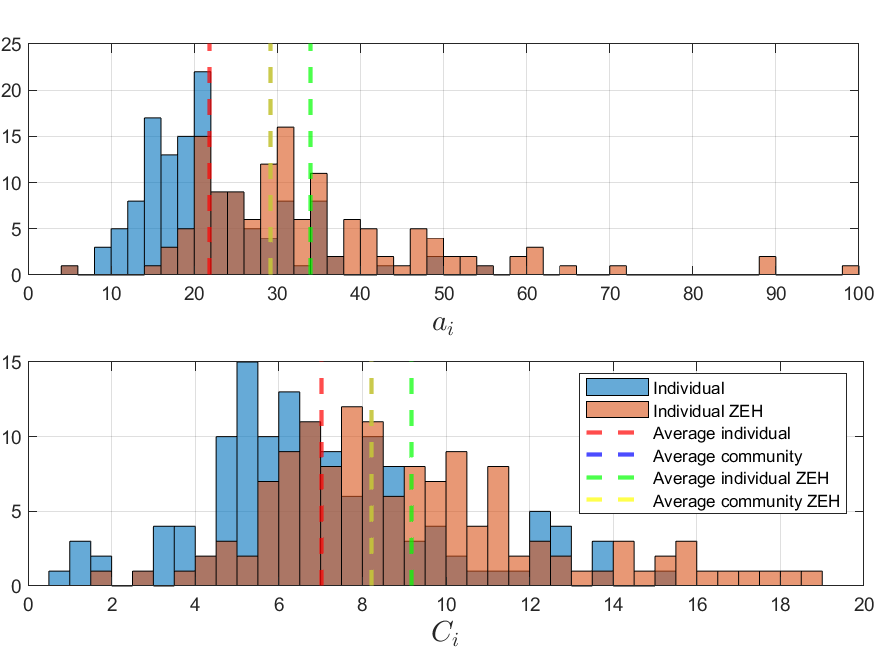}
    \caption{Histograms of investments on PV panels (on the top) and batteries (on the bottom) for the penalized reverse power scenario}.
    \label{fig:histogram_4500_10}
\end{figure}
Fig.~\ref{fig:cost_4500_10} shows the total savings after optimization with respect to the  cost before optimization, i.e. when no PV panels nor batteries are installed. As it can be expected, for the problem without the ZEH constraint, the optimized costs are smaller than the previous costs for every user. However, some users face larger costs than before optimization (those with negative savings) when considering the ZEH constraint. This can be seen as a consequence of a low capability of generating solar energy for some households along with the reverse power penalty, which makes very difficult to attain ZEH.
\begin{figure*}[thb]
    \centering
    \includegraphics[width=0.95\textwidth]{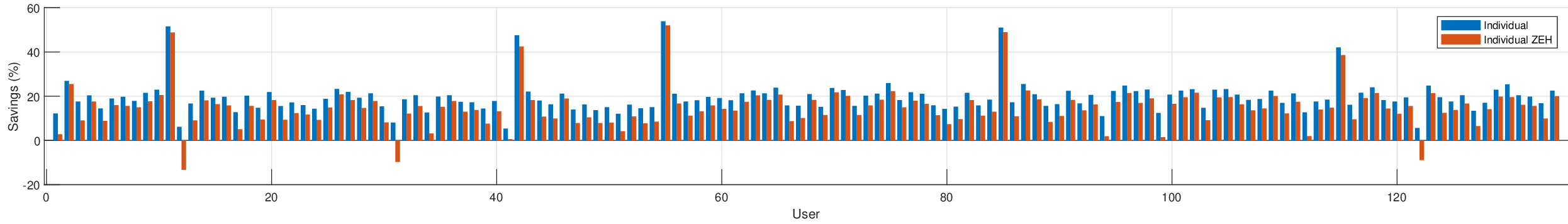}
    \caption{Improvement of the cost of each user with respect to the case where batteries and PV panels are not installed for the penalized reverse power scenario.}
    \label{fig:cost_4500_10}
\end{figure*}
Finally, Table \ref{table:4500_rev10} shows the summarized results for the proposed set of parameters. Note that ``Ind.'' refers to ``Individual'', ``Ind. ZEH'' corresponds to the individual optimization problem including the ZEH constraint, ``Sharing'' refers to the  investment approach as a group of users, ``Sharing ZEH'' corresponds to the same investment approach including the ZEH constraint, ``Av. PV'' refers to the average amount of invested PV among users, ``Av. battery'' stands for the average battery size among users, ``ZEH (\%)'' refers to the percentage of users achieving ZEH in the original optimization problem and ``savings'' stands for the improvement of the cost with respect to the case where no batteries nor PV panels are installed. 

It can be seen that neither the individual users nor the community achieve ZEH in the original problem. This happens due to the penalty in the reverse power flow which encourages users not to invest a lot on PV panels. As attaining ZEH is not optimal cost-wise, users do not invest beyond needed in order to attain ZEH. Also, it seems that there are some users who cannot achieve ZEH on their own. This might be because certain buildings are oriented towards a direction where solar radiation is very low and thus $a_{\text{max}}$ $\text{m}^{2}$ worth of solar panels are not enough in order to satisfy the ZEH constraint, leading to infeasibility. Also, note that it is always easier to achieve ZEH as a community instead of individually, i.e. a smaller amount of PV panels is required.
 \begin{table}[htbp]
     \caption{Summarized results for the penalized reverse power scenario.}
 \begin{center}
 \begin{tabular}{|c|c|c|c|c|c|}
 \hline
  &  Ind. & Ind. ZEH & Sharing & Sharing ZEH  \\
 \hline \hline
 Av. PV ($\text{m}^{2}$) & 21.90 & 33.97 & 20.65 & 28.33   \\ \hline
 Av. battery (kWh) & 7.015 & 9.164 & 8.022 & 9.025   \\  \hline
 ZEH (\%) & 0 & 99.25 & 0 & 100 \\ \hline
 Savings (\%) & 19.64 & 14.72 & 25.29 & 21.90 \\ \hline
 \end{tabular}
 \end{center}
   \label{table:4500_rev10}
 \end{table} 

\subsection{Scenario 2. Non-penalized reverse power}
Here, it is considered that the amount of electricity injected back to the grid does not affect the value of the cost function, that is $\Pi_{R} = 0 \; \yen /\text{kWh}$. It can be considered as the intermediate case between the penalized reverse power case and the FiT case.
\begin{figure}[tbh]
    \centering
    \includegraphics[width=0.48\textwidth]{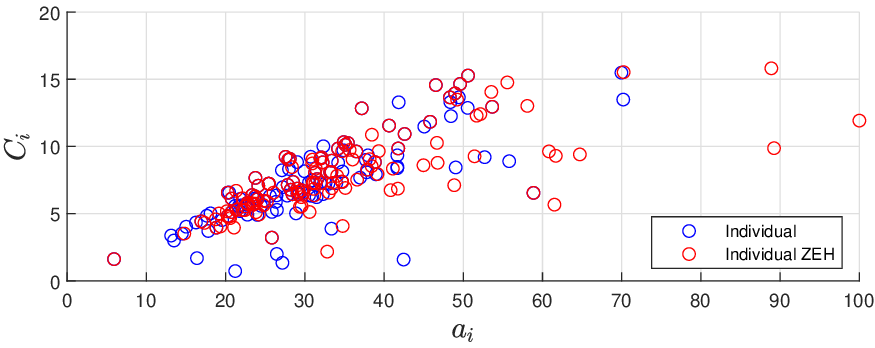}
    \caption{PV panels and batteries investments with and without ZEH constraints for the non-penalized reverse power scenario.}
    \label{fig:aC_4500_0}
\end{figure}
\begin{figure}[tbh]
    \centering
    \includegraphics[width=0.48\textwidth]{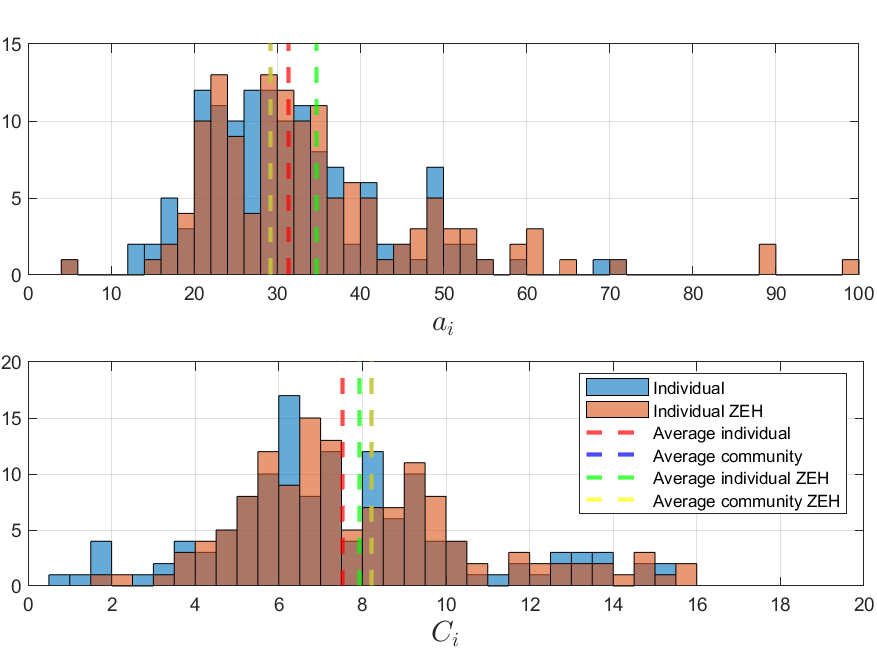}
    \caption{Histograms of investments on PV panels (on the top) and batteries (on the bottom) for the non-penalized reverse power scenario.}
    \label{fig:histogram_4500_0}
\end{figure}
The results are shown in Fig.~\ref{fig:aC_4500_0}, \ref{fig:histogram_4500_0}, \ref{fig:cost_4500_0} and Table \ref{table:4500_rev0}. In Fig.~\ref{fig:aC_4500_0}, the investment in PV panels and batteries is shown for each user. It is easy to see that the investment is larger in comparison to that of Scenario 1, i.e. 30.52\% for the PV investment and 7.290\% for the battery investment on average in the individual case whereas, for the community case, the PV investment increases by 29.46\% and the battery investment increases by 3.820\%. This happens because we are not penalizing the reverse power flow and thus more PV panels can be installed without increasing the costs due to an excessive generation. Also, it can be seen that, in order to attain ZEH, many individual users must increase the amount of PV panels, but it is not as pronounced as in Scenario 1. Again, The ZEH constraint requires an expansion of PV area by 7.487\% and a purchase of 5.493\% additional battery storage capacity on average for the individual case. As the community attains ZEH naturally, there is no change in adding the ZEH constraint to the problem. Similarly, Fig.~\ref{fig:histogram_4500_0} shows the investment results in a histogram manner. It can be seen that the distribution of the PV investment is slightly pushed towards the right direction in comparison with Fig.~\ref{fig:histogram_4500_10}, which makes sense due to the lack of penalty in the reverse flow. However, the distribution of the investment related to the battery does not change substantially. 
\begin{figure*}[thb]
    \centering
    \includegraphics[width=0.95\textwidth]{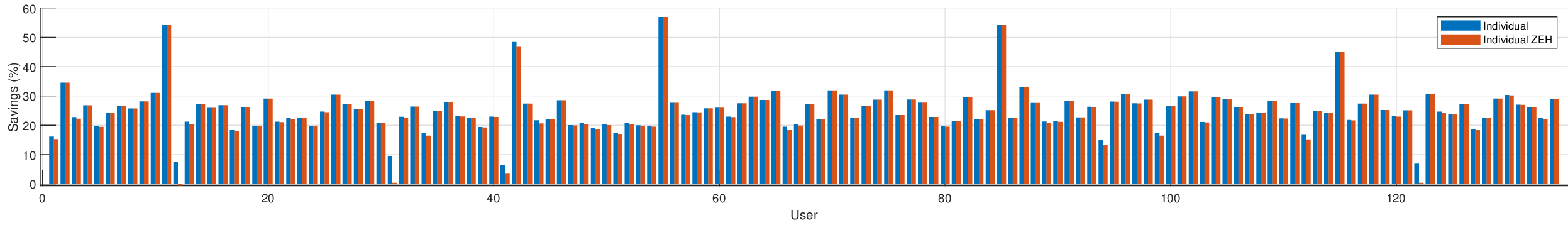}
    \caption{Improvement of the cost of each user with respect to the case where batteries and PV panels are not installed for the non-penalized reverse power scenario.}
    \label{fig:cost_4500_0}
\end{figure*}
Finally, Fig.~\ref{fig:cost_4500_0} shows the individual costs for the proposed set of parameters. In general, it can be seen that the savings are larger than the savings depicted by Fig.~\ref{fig:cost_4500_10}, which is obtained from the set of parameters in Scenario 1. Also, it is easy to see that the difference in the savings when including the ZEH constraint is really small, leading to the conclusion that attaining ZEH is not as hard as we might think. However, there are some exceptions that should be studied on a case-by-case basis.

Again, Table \ref{table:4500_rev0} shows the summarized results for the proposed  set of parameters. It can be seen that a high percentage of users are still not achieving ZEH naturally. 
However, in the sharing approach, ZEH is attained in the original investment problem. This means that even though achieving ZEH individually can be tough, it becomes easier when considering a sharing-based approach.
 \begin{table}[htbp]
     \caption{Summarized results for the non-penalized reverse power scenario.}
 \begin{center}
 \begin{tabular}{|c|c|c|c|c|c|}
 \hline
  &  Ind. & Ind. ZEH & Sharing & Sharing ZEH  \\
 \hline \hline
 Av. PV ($\text{m}^{2}$) & 31.36 & 34.76 & 29.19 & 29.19  \\ \hline
 Av. battery (kWh) & 7.524 & 7.937 & 8.206 & 8.206  \\  \hline
 ZEH (\%) & 43.28 & 99.25 & 100 & 100  \\ \hline
 Savings (\%) & 25.35 & 25.00 & 31.27 & 31.27  \\ \hline
 \end{tabular}
 \end{center}
   \label{table:4500_rev0}
 \end{table} 

\subsection{Scenario 3. Feed-in tariffs}
In this case, users can make a profit by returning the surplus power generated by means of the PV panels to the grid, that is $\Pi_{R} = -5 \; \yen /\text{kWh}$.

The results are shown in Fig.~\ref{fig:aC_4500_-5}, \ref{fig:histogram_4500_-5}, \ref{fig:cost_4500_-5} and Table~\ref{table:4500_rev-5}. Fig.~\ref{fig:aC_4500_-5} shows the investment for the PV panels and batteries for each individual user. It can be seen that the solution to the original optimization problem is almost the same as the solution to the problem including ZEH enforcing constraint, which means that many users attain ZEH in the original optimization problem. Particularly, it is only needed to increase the PV investment by 1.719\% and the battery investment by 2.426\% on average.
Also, as can be expected, the amount of investment in PV is significantly larger than before due to the fact that now it is possible to make a profit by returning this surplus energy to the grid. This can be seen clearly in Fig.~\ref{fig:histogram_4500_-5}, where the average value of the PV investment is over twice in comparison to Scenario 2, and the distribution is more flattened and pushed to the right. On the other hand, the distribution of the battery size is almost identical to the previous two scenarios.
\begin{figure}[tbh]
    \centering
    \includegraphics[width=0.48\textwidth]{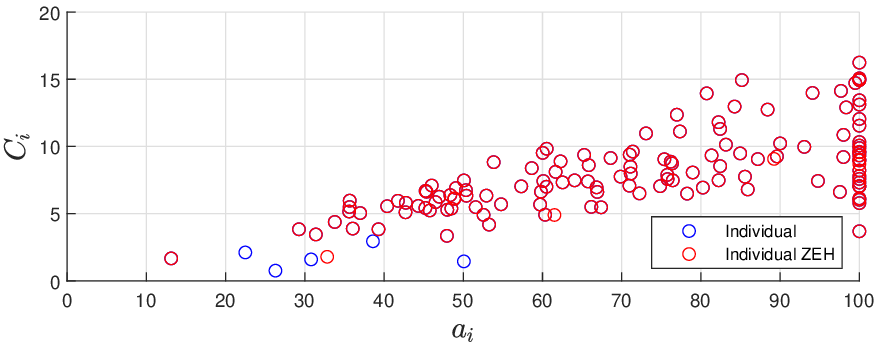}
    \caption{PV panels and batteries investments with and without ZEH constraints for the FiT scenario.}
    \label{fig:aC_4500_-5}
\end{figure}

\begin{figure}[tbh]
    \centering
    \includegraphics[width=0.48\textwidth]{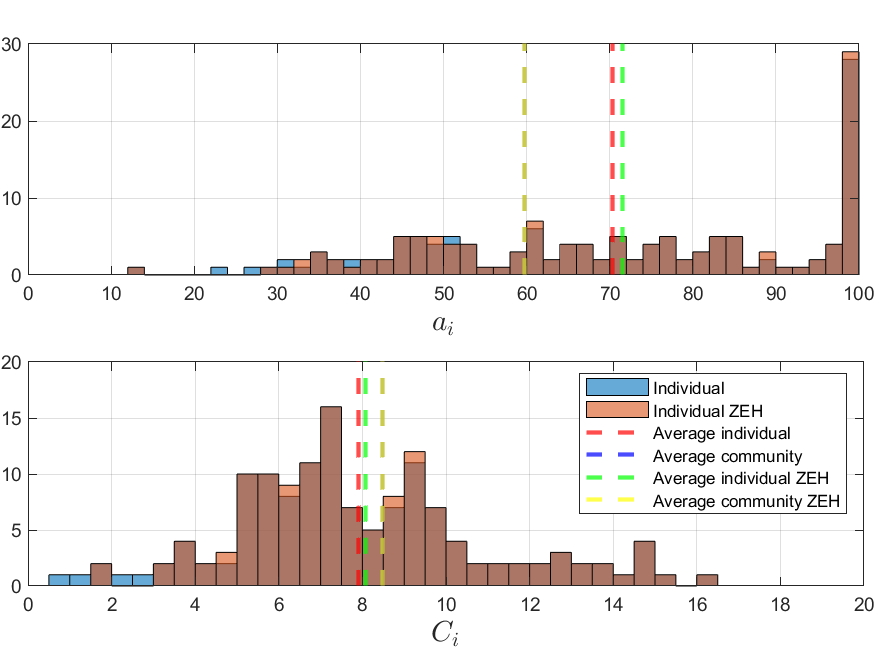}
    \caption{Histograms of investments on PV panels (on the top) and batteries (on the bottom) for the FiT scenario.}
    \label{fig:histogram_4500_-5}
\end{figure}




\begin{figure*}[tbh]
    \centering
    \includegraphics[width=0.95\textwidth]{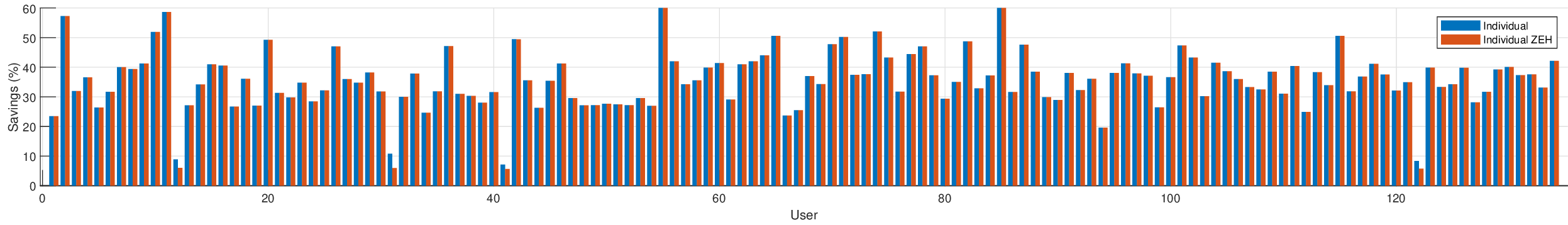}
    \caption{Improvement of the cost of each user with respect to the case where batteries and PV panels are not installed for the FiT scenario.}
    \label{fig:cost_4500_-5}
\end{figure*}

Finally, Table~\ref{table:4500_rev-5} shows the summarized results for the proposed set of parameters. As it can be expected, the amount of savings that can be attained are higher than in the two previous settings since instead of penalizing the reverse power flow, users can make a profit of their spare energy. Also, it can be seen that almost all users attain ZEH naturally. Only a small percentage of users do not achieve ZEH in the original optimization problem. This might be related to the fact that the solar radiation of these houses is scarce and thus it is very expensive to generate energy by using these means. 
 \begin{table}[htbp]
\caption{Summarized results for the FiT scenario.}
 \begin{center}
 \begin{tabular}{|c|c|c|c|c|c|}
 \hline
  &  Ind. & Ind. ZEH & Sharing & Sharing ZEH  \\
 \hline \hline
 Av. PV ($\text{m}^{2}$) & 70.36 & 71.59 & 59.70 & 59.70  \\ \hline
 Av. battery (kWh) & 7.903 & 8.065 & 8.493 & 8.493  \\  \hline
 ZEH (\%) & 96.27 & 99.25 & 100 & 100 \\ \hline
 Savings (\%) & 35.74 & 35.65 & 41.38 & 41.38  \\ \hline
 \end{tabular}
 \end{center}
   \label{table:4500_rev-5}
 \end{table} 

\subsection{Discussion}

As expected, various reverse power policies encourage different investment behaviors depending on the chosen value of $\Pi_{R}$. When considered as a penalty, it is opposite to our objective, which is attaining ZEH. In order to avoid the reverse power penalty, users tend not to buy a large quantity of PV panels. When considering the sharing-based approach, the situation remains the same. Thus, a different kind of mechanism should be considered to tackle the problem of reverse power as this approach prevents users from attaining ZEH.

On the opposite side, when considering FiT, the amount of PV becomes too large. See Fig.~\ref{fig:phi_discussion}, where the sum of $\phi_{k}^{+}$ and $\phi_{k}^{-}$ for all users (individual approach) during April when $\Pi_{R} = 10$, $\Pi_{R} = 0$ and $\Pi_{R} = -5$ is shown. There, it can be seen the magnitude of the reverse power flow and the amount of needed gas for every considered value of $\Pi_{R}$. Also, Table~\ref{table:discussion_phipl} and  \ref{table:discussion_phim} show the average amount of $\phi_k^{+}$ and $\phi_k^{-}$. Again, it can be seen that the amount of reverse power is much greater when considering FiT. Then, we have the following dilemma: if we penalize the reverse power flow, users will not achieve ZEH, but if we encourage the acquisition of PV by means of FiT, then the amount of reverse power is excessively high and it might lead to malfunctions in the grid. As an intermediate case, we have $\Pi_{R} = 0$, where the amount of PV is considerably higher than $\Pi_{R} = 10$, but not as high as when considering $\Pi_{R} = -5$. Some users can achieve ZEH on their own, but it is only a small minority. However, when they work together in the sharing approach, it is possible to attain ZEH as a group (which was even impossible when $\Pi_{R} = 10$). 

\begin{figure}[tbh]
    \centering
    \includegraphics[width=0.48\textwidth]{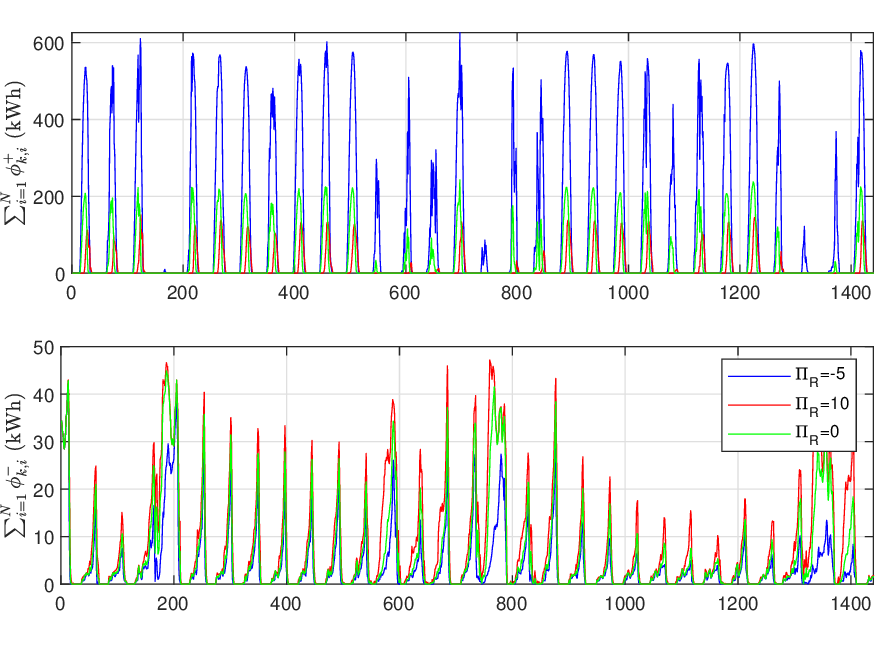}
    \caption{Total amount of $\phi_k^{+}$ (on the top) and total amount of $\phi_k^{-}$ (at the bottom) during April.}
    \label{fig:phi_discussion}
\end{figure}

 \begin{table}[htbp]
     \caption{Average $\phi_k^{+}$ per user every 30 minutes.}
 \begin{center}
 \begin{tabular}{|c|c|c|c|}
 \hline
  &  $\Pi_{R}=-5$ & $\Pi_{R}=0$ & $\Pi_{R}=10$  \\
 \hline \hline
Individual & 0.4945 & 0.1017 & 0.0346  \\ \hline
Community & 0.3998 & 0.1079 & 0.0366 \\ \hline
 \end{tabular}
 \end{center}
   \label{table:discussion_phipl}
 \end{table} 
 
  \begin{table}[htbp]
      \caption{Average $\phi_k^{-}$ per user every 30 minutes.}
 \begin{center}
 \begin{tabular}{|c|c|c|c|}
 \hline
  &  $\Pi_{R}=-5$ & $\Pi_{R}=0$ & $\Pi_{R}=10$  \\
 \hline \hline
Individual & 0.0826 & 0.1196 & 0.1517  \\ \hline
Community & 0.0633 & 0.0974 & 0.1251 \\ \hline
 \end{tabular}
 \end{center}
   \label{table:discussion_phim}
 \end{table} 

In summary, according to the data of the consumption and solar generation,  the right incentive for this case would be to choose $\Pi_{R}=0$ and make users cooperate. This would attain ZEH 
while at the same time preventing the power grid from being subjected to an excessive amount of reverse power. In order to alleviate the effect of the reverse power flow, an appropriate online management of the battery can be considered. That is, one could try to make the reverse power flow as constant as possible by deciding  when the battery should charge or discharge. 

\subsection{Computation times}
In order to emphasize the contribution of the paper and the importance of the proposed LP problem, we implement the original problem
\eqref{eq:original_after_remark} in Gurobi to show that even commercial solvers have trouble solving the original problems. 

We compare the computational time needed for solving problem \eqref{eq:original_after_remark} in Gurobi and the reformulated LP problem \eqref{eq:num_example_individual} using linprog on MATLAB with different values of $K$.
Note that the price of the PV and the battery are considered to be amortized during these periods of time so that the problem still makes sense. The comparison can be seen in Fig.~\ref{fig:time_comparison}. It is shown that
the LP problem is solved much faster than the original problem. Specifically,
when considering 58 days (about two months), the original problem becomes computationally challenging, taking up to 14 hours to be solved. In contrast, the full LP problem can be solved in less than ten seconds. 
Therefore, given the extensive scale of the original problem with large $K$, it cannot be considered tractable in practice.
\begin{figure}[tbh]
    \centering
    \includegraphics[width=0.48\textwidth]{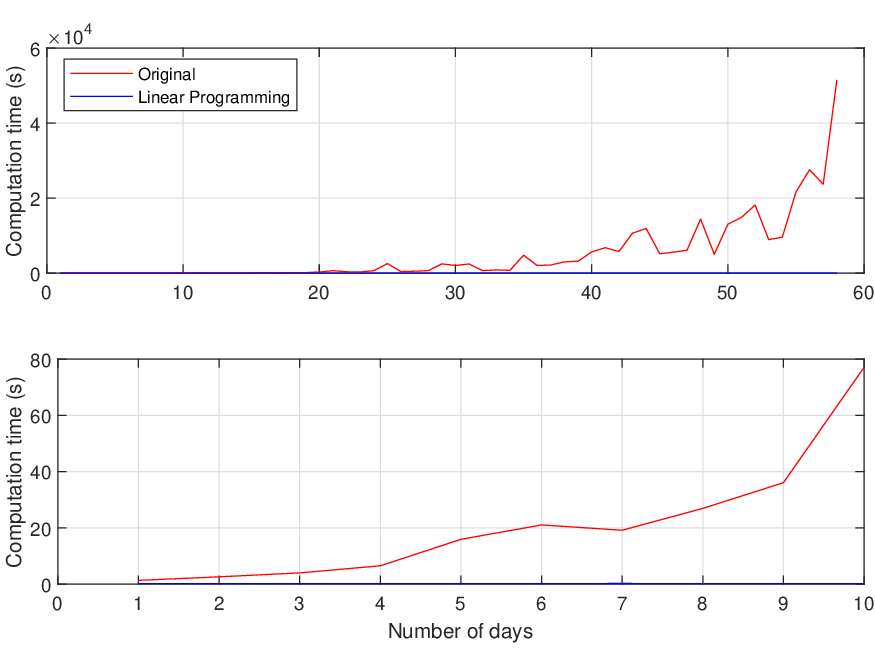}
    \caption{On the top: computation times up to approximately two months. On the bottom: computation times of the first 10 days.} 
    \label{fig:time_comparison}
\end{figure}

\section{Conclusion}
\label{sec:conclusions}
In this paper, we proposed optimization problems to optimally decide the amount of PV panels and the size of the battery to be installed in a house in order to save money in electricity bills and, more importantly, achieve ZEH. Concretely, the main highlights of the paper were:
\begin{itemize}
    \item We have shown the potential advantages of the sharing economy in comparison to a strictly individual investment approach.
    \item It has been shown that the nonconvex optimization problems associated with the optimal sizing in this paper can be posed as LP problems and thus they can be easily solved. Also, they are not computationally heavy, compared to other existing methods in the literature.
    \item Through real data obtained from a neighborhood in Japan, we have shown in the numerical examples how the optimization problems perform for three sets of parameters corresponding to different incentive scenarios. 
    \item Finally, it has been
    shown that a sharing-based approach without penalty in the reverse power can achieve ZEH while obtaining a low reverse power flow. 
\end{itemize}
On the other hand, it should be noted that
the optimization problem only considers data for a specific year, which might lead to misleading results if the year is not representative of actual household consumption and generation.
This could be solved by integrating stochastic optimization approaches with the proposed LP problem. 
Therefore, we consider as future work  
the improvement for the aforementioned investment problem, the development of efficient management strategies for the battery to minimize the effect of the reverse power in the grid, and different sharing strategies.

\appendices
\section{}
In this appendix, we test the suitability of the model in section \ref{sec:single}, that is
\begin{equation}
C_{k} =   \left[ C_k^{+} \right]^{\bar{C}}_{0}   :=
    \begin{cases}
        0 & \text{if } C_{k}^{+} \leq 0 \\
        \bar{C} & \text{if }  C_{k}^{+}  \geq \bar{C} \\
        C_{k}^{+} & \text{else} 
    \end{cases}  
\end{equation}
where
\begin{equation}
C_{k}^{+} = \gamma \,C_{k-1} + a Y_{k-1} - X_{k-1}, \; \gamma = 0.99998.
\end{equation}
For this purpose, the model is compared against a MATLAB Simscape battery block. This battery corresponds to a lithium-ion battery with a nominal voltage of $12.6$V, rated capacity of $675$Ah and a battery response time of $90$s.
In this experiment, we inject steps of different amplitude and length to both the Simulink system and the proposed model for a total of 400 hours (approximately 17 days). 

The results are shown in Fig.~\ref{fig:model_comparison}. There, it can be seen that the proposed model is obtaining a good approximation of the real system, attaining a mean squared error of $1.2542$, which can be considered low taking into account the order of magnitude of the SoC. Thus, the model is valid for the investment problem.
\begin{figure}[tbh]
    \centering
    \includegraphics[width=0.48\textwidth]{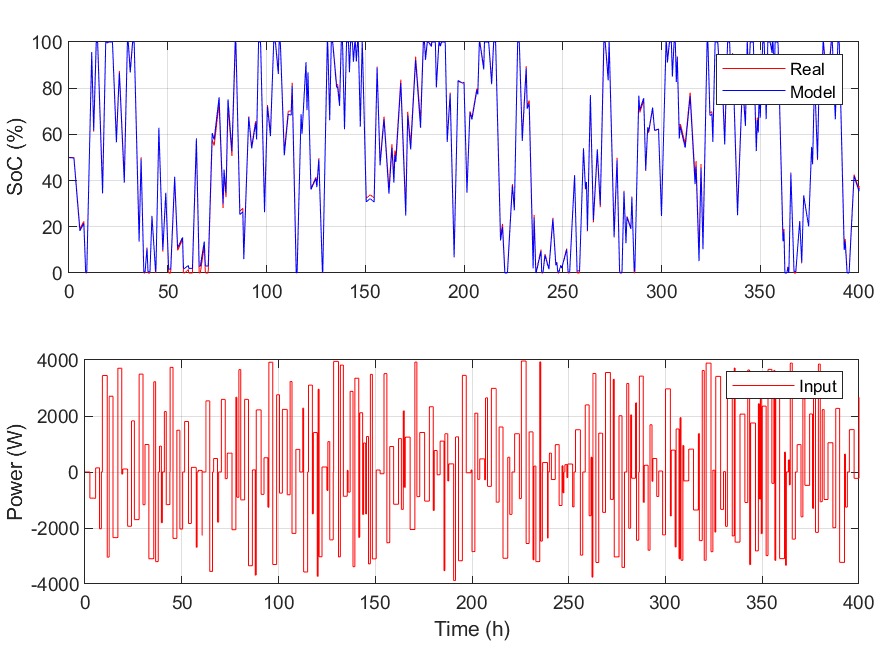}
    \caption{On the top: real SoC and prediction obtained by means of the proposed model. On the bottom: applied random input.}
    \label{fig:model_comparison}
\end{figure}

\bibliographystyle{IEEEtran}
\bibliography{Bibl}

\EOD

\end{document}